\documentclass[12pt,reqno]{amsart}

\usepackage{amsfonts}
\usepackage{graphicx}
\usepackage{tabularx}
\usepackage{array}
\usepackage[usenames,dvipsnames]{color}
\usepackage{comment}
\usepackage{amsmath,amsthm,amssymb}
\usepackage{fullpage}
\usepackage{xcolor}
\usepackage{tikz}
\DeclareMathOperator{\dist}{dist}

\tikzstyle{legend_general}=[rectangle, rounded corners, thin,
                          top color= white,bottom color=lavander!25,
                          minimum width=2.5cm, minimum height=0.8cm,
                          violet]

\newtheorem{theorem}{Theorem}[section]

\newtheorem{conjecture}[theorem]{Conjecture}

\newtheorem{lemma}[theorem]{Lemma}
\newtheorem{corollary}[theorem]{Corollary}
\theoremstyle{definition}
\newtheorem{definition}[theorem]{Definition}

\title{On the cop number of graphs of high girth}

\author{Peter Bradshaw}
\address{Department of Mathematics, Simon Fraser University, Burnaby, Canada}
\email{pabradsh@sfu.ca}

\author{Seyyed Aliasghar Hosseini}
\address{Department of Mathematics, Simon Fraser University, Burnaby, Canada}
\email{sahossei@sfu.ca}

\author{Bojan Mohar$^*$}
\address{Department of Mathematics, Simon Fraser University, Burnaby, Canada}
\email{mohar@sfu.ca}
\thanks{$^*$Supported in part by the NSERC Discovery Grant R611450 (Canada), 
and by the Research Project J1-8130 of ARRS (Slovenia).}
\thanks{$^*$On leave from IMFM, Department of Mathematics, University of Ljubljana.}

\author{Ladislav Stacho$^\dag$}
\address{Department of Mathematics, Simon Fraser University, Burnaby, Canada}
\email{lstacho@sfu.ca}
\thanks{$^\dag$ Supported in part by the NSERC Discovery Grant R611368 (Canada)}

\begin{document}

\begin{abstract}
We establish a lower bound for the cop number of graphs of high girth in terms of the minimum degree, and more generally, in terms of a certain growth condition. We show, in particular, that the cop number of any graph with girth $g$ and minimum degree $\delta$ is at least $\tfrac{1}{g}(\delta - 1)^{\lfloor \frac{g-1}{4}\rfloor}$. We establish similar results for directed graphs. While exposing several reasons for conjecturing that the exponent $\tfrac{1}{4}g$ in this lower bound cannot be improved to $(\tfrac{1}{4}+\varepsilon)g$, we are also able to prove that it cannot be increased beyond $\frac{3}{8}g$. This is established by considering a certain family of Ramanujan graphs. In our proof of this bound, we also show that the ``weak" Meyniel's conjecture holds for expander graph families of bounded degree.
\end{abstract}

\maketitle

\section{Introduction}

We consider the \emph{game of cops and robbers}, a two-player game played with perfect information on a finite connected graph $G$. In the game, the first player controls a team of \emph{cops}, and the second player controls a \emph{robber}. At the beginning, the first player places each cop on her team at a vertex of $G$, and then the second player places the robber at a vertex of $G$. The two players take turns. On the first player's turn, she may move each cop $C$ to a vertex in the closed neighborhood of $C$, and on the second player's turn, he may move the robber to any vertex in the closed neighborhood of the robber's current vertex. The first player wins the game if a cop moves to the same vertex as the robber; in this case, the cop \emph{captures} the robber. The second player wins by letting the robber avoid capture indefinitely. In order to establish a finite-time win condition, we may also say that the second player wins the game if the same game position occurs twice; this change makes the game finite and does not affect the strategy of either player.

The game of cops and robbers was introduced by Quilliot \cite{Quilliot}, and independently by Nowakowski and Winkler \cite{Nowakowski}. Aigner and Fromme \cite{Aigner} introduced the concept of the \textit{cop number} $c(G)$ of a graph $G$, which denotes the number of cops that the first player needs to capture the robber on $G$ with optimal play. Upper bounds on cop number are well-understood for many classes of graphs. For instance, planar graphs and toroidal graphs have cop number at most $3$ \cite{Aigner,LehnerTorus}. Recently, Bowler et al.\ \cite{Bowler} proved that the cop number of a graph of genus $g$ is at most $\frac{4}{3}g + \frac{10}{3}$, thus improving an earlier bound of $\frac{3}{2}g + 3$ by Schr\"oder \cite{Schroeder}. In a similar flavor, Andreae \cite{Andreae} proved that excluding a graph $H$ with $t$ edges as a minor in a graph $G$ implies that $c(G)\le t$. There are also cop number bounds for highly symmetrical graphs. For instance, Frankl \cite{FranklGirth} showed that normal Cayley graphs of degree $d$ have cop number at most $d$. More recently, it was shown that abelian Cayley graphs on $n$ vertices have cop number at most $0.95\sqrt{n} + 2$ \cite{Jeremie}. Lower bounds for cop number are also known for several graph classes. For instance, projective plane incidence graphs with $2q^2 + 2q + 2$ vertices have cop number at least $q+1$ \cite{Baird}, and certain abelian Cayley graphs on $n$ vertices have cop number $\Omega(\sqrt{n})$, with some families achieving their cop number as high as $\frac{1}{2} \sqrt{n}$ \cite{Hasiri,Jeremie}.
Furthermore, Bollob\'as et al.~\cite{BollobasKun} and Pra\l at and Wormald \cite{PraWor2} showed that random graphs in ${\mathcal G}(n,p)$ have cop number of order $\Theta(\sqrt{n})$ a.a.s. Finally, for sparse random graphs (random $d$-regular graphs with $d$ constant), an upper bound of $O(\sqrt n)$ also holds a.a.s.~\cite{PraWor1}.

In this paper, we focus on lower bounds for the cop number of graphs of high girth. Aigner and Fromme \cite{Aigner} proved that graphs with girth at least $5$ and minimum degree $\delta$ have cop number at least $\delta$. Their ideas were extended later by Frankl \cite{FranklGirth}, who showed more generally that graphs with girth at least $8t - 3$ have cop number at least $(\delta - 1)^t + 1$. We seek to improve these lowers bounds. Our main result, proved in Section \ref{sectionUndirected}, is the following lower bound for cop number in terms of a graph's girth $g$ and minimum degree $\delta$. 

\begin{theorem}
Let $t \geq 1$ be an integer, and let $G$ be a graph of girth $g \geq 4t + 1$ and with minimum degree $\delta$. Then $c(G) > \frac{1}{e\,t} (\delta-1)^{t}$.
\label{thmWeight}
\end{theorem}

Our bound is of the form $\Omega(g^{-1}(\delta-1)^{g/4})$, which is asymptotically superior to the afore-mentioned bound of Frankl \cite{FranklGirth}, which is of the form $\Omega((\delta-1)^{g/8})$. One may naturally ask if the coefficient $\frac{1}{4}$ of $g$ is best possible. While we do not have an answer to this question, we note that by assuming two well-known conjectures, we would be able to argue that $\frac{1}{4}$ is best possible. 

The following fundamental conjecture about cop number appeared in a paper by Frankl \cite{FranklGirth} where it was referred to as a personal communication to the author by Henri Meyniel in 1985 (see also \cite{Baird}).

\begin{conjecture}[Meyniel's Conjecture]
\label{conjMeyniel}
For every connected graph $G$ of order $n$, $c(G)\le O(\sqrt{n}\,)$.
\end{conjecture}

In \cite{Bollobas}, Bollob\'{a}s and Szemer\'{e}di conjectured that for every sufficiently large $g$, there exists a cubic graph of girth $g$ with $\Theta(2^{(\frac{1}{2}+o(1))g})$ vertices. The corresponding speculation for every fixed degree $d$ is also a folklore conjecture that is related to the Moore bound.

\begin{conjecture}[Folklore]
\label{conjMoore}
For any positive integers $g$ and $d$, there exists a $d$-regular graph $G$ of girth at least $g$, and order $\Theta((d-1)^{(\frac{1}{2}+o(1))g})$.
\end{conjecture}

If we assume that a family of graphs with $\Theta((d-1)^{(\frac{1}{2}+o(1))g})$ vertices exists, and if we assume that Meyniel's conjecture holds for this family of graphs, then this family of graphs must have cop number $O((d-1)^{(\frac{1}{4}+o(1))g})$. It would then follow that for any lower bound of the form $\Omega((d-1)^{c g})$ for the cop number of all graphs of girth $g$ and minimum degree $d$, $c \leq \frac{1}{4}$, making our lower bound of $\frac{1}{4}$ best possible. 

We use the same techniques as in the proof of Theorem \ref{thmWeight} to show that a similar lower bound applies to graphs of high girth with a certain growth condition, even when the minimum degree may be small. We state this lower bound in Theorem \ref{thmGrowth}. In Section \ref{sectionDirected}, we apply the techniques of Section \ref{sectionUndirected} to directed graphs. The corresponding statements are given as Theorems \ref{thmOutDeg} and \ref{thmFrankl}. An important observation from our proof of Theorem \ref{thmOutDeg} is that it is not girth that is important for the cop number of a digraph, but rather the novel notion of the ``trap distance'' $\rho^*(u,v)$ between the vertices of a digraph.

Finally, we compute an upper bound on the cop number of a family of Ramanujan graphs in terms of their degree $d$ and girth $g$. We state this bound in Theorem \ref{thm:3/8}. This upper bound, which is of order $O((d-1)^{(\frac{3}{8}+o(1))g})$, compares quite favorably with our lower bound in Theorem~\ref{thmWeight}. The proof method used in establishing this upper bound works more generally for arbitrary expander graphs.

\section{Lower bounds for undirected graphs}
\label{sectionUndirected}

In this section, we derive general lower bounds on the cop number of graphs of high girth. This bound significantly improves the previous best lower bound of Frankl \cite{FranklGirth}, which was obtained in 1987. 

\begin{proof}[Proof of Theorem \ref{thmWeight}]
If $t = 1$, then the theorem follows from a result of Aigner and Fromme which states that for graphs of girth at least $5$, the cop number of a graph is at least its minimum degree \cite{Aigner}. Otherwise, we assume that $t \geq 2$.

We will show that under the stated girth and degree conditions, the robber has a winning strategy, provided that the number of cops is at most $\frac{1}{e\,t} (\delta-1)^{t}$. We observe that our girth condition implies that for any two vertices $u,v$ at distance at most $2t$, there is a unique geodesic joining $u$ to $v$. 

We begin the game on the robber's move with all cops on a single vertex $v_0 \in V(G)$. (In fact, we view the first move of the robber as part of his choice of the initial position.) We let the robber begin at a vertex $v_1$ adjacent to $v_0$. To show that the cops have no winning strategy on $G$, it suffices to show that the cops cannot win from this position. At any given state $s$ of the game with robber to move, we assume that the robber is at a vertex $v_s$ and that he will move to a neighboring vertex $v_{s+1}$ different from $v_{s-1}$ (and different from $v_s$). After the cops move, we will reach the next state $s+1$.  We will show that the robber can move to a vertex $v_{s+1}$ for which no cop is positioned at a vertex in the closed neighborhood $N[v_{s+1}]$. Thus the robber will be able to avoid capture during state $s$, and hence forever. 

We define the following values. Let $q=\delta-1$, $r=(1-\frac{1}{t})q$, and let $K \leq \frac{1}{e\,t} q^{t}$ denote the number of cops. The lower bound of the theorem is obviously true if $\delta \le 2$. Thus we may assume that $\delta\ge3$.

Suppose that the game is at a state $s\geq 1$. Let $u_1,\dots, u_q$ be distinct neighbors of $v_s$ that are different from $v_{s-1}$. (There may be additional neighbors of $v_s$ if $\deg(v_s)>\delta$, but we do not need them for our strategy.) For each $i=1,\dots,q$, let ${\mathcal C}_i$ be the set of cops $C$ such that $\dist(C,v_s) \leq 2t $ and the geodesic from $C$ to $v_s$ passes through $u_i$. (Note that since $C$ is not at $u_i$, such a geodesic cannot pass via $v_{s-1}$.) For such a cop $C\in \bigcup_{i = 1}^q {\mathcal C}_i$, we let $\rho = \dist(C,v_s)$, and we define the \textit{weight} of $C$ as $w(C) = r^{t-\lceil \rho/2\rceil}$. We let $k$ be the number of cops that are in none of the sets ${\mathcal C}_1,\dots,{\mathcal C}_q$, and we define the weight of each such cop to be equal to $1$. Let $W_i = \frac{k}{q} + \sum_{C\in {\mathcal C}_i} w(C)$, and let $W$ be the total weight of all cops. Note that $W = \sum_{i=1}^q W_i$. 

At the beginning of each state $s$, the robber selects a neighbor $u_j$ for which $W_j$ is minimum and moves to $u_j$. If $W_j < r^{t-1}$, then neither $u_j$ nor any of its neighbors contains a cop, so $u_j$ is a safe vertex for the robber. In order to prove that the robber is never captured, it suffices to prove that $W < qr^{t-1}$, since in that case, $W_j < r^{t-1}$ for some $1 \leq j \leq q$.

Initially, ${\mathcal C}_1,\dots,{\mathcal C}_q$ are all empty (since all cops are at $y_1$ and the girth of $G$ is greater than $2t+1$), and hence all cops have weight $1$. Thus $W < qr^{t-1}$ as long as we have fewer than $qr^{t-1}$ cops. It is easy to see by applying the inequality $\frac{1}{e} < (1-\frac{1}{t})^{t-1}$ that this is the case: 
\begin{equation}
\label{eq:1}
   K \leq \frac{1}{e\,t} q^{t} < \frac{(1-\frac{1}{t})^{t-1}}{t} \, q^{t} = \frac{qr^{t-1}}{t} < qr^{t-1}.
\end{equation}

The rest of the proof is by induction. We assume that at the current state $s$ we have $W < qr^{t-1}$, and we let $W'$ denote the total weight of the cops after the robber has moved to $v_{s+1}=u_j$ and state $s$ has ended. For every cop $C \in {\mathcal C}_j$, the robber moved closer to $C$, and $C$ may have moved closer to the robber; thus the distance $\rho$ between $C$ and the robber may have decreased by at most $2$. Consequently, the new weight $w'(C)$ is at most $rw(C)$. By the girth condition of $G$, at the end of state $s$, for any other cop $C'$ that was not in $\mathcal{C}_j$, either the geodesic from $C'$ to $v_{s+1}$ passes through $v_s$, or $\dist(C', v_{s+1}) \geq 2t - 1$. In both cases, $C$' has a weight equal to $1$ at the end of state $s$.  Then, using (\ref{eq:1}), we see that
\begin{eqnarray}
   W' & \leq & r\cdot W_j + K \nonumber \\
   & < & r\cdot \frac{W}{q} + \frac{qr^{t-1}}{t} \nonumber \\
   & < & r^{t} + \frac{qr^{t-1}}{t} \label{eq:W'} \\
   & = & qr^{t-1} \Bigl(1 - \frac{1}{t} + \frac{1}{t}\Bigr) = qr^{t-1}.\nonumber
\end{eqnarray}
This completes induction and the proof.
\end{proof}

Theorem \ref{thmWeight} gives a lower bound for the cop number of a graph of girth $g$ and minimum degree $\delta$ of the form $\Theta (g^{-1} (\delta - 1)^{g/4})$, which is a significant improvement over Frankl's lower bound from \cite{FranklGirth}, which is of the form $\Theta((\delta - 1)^{g/8})$. Furthermore, the method of Theorem \ref{thmWeight} extends to graphs which have a fast growth property, which we define next. 

Given positive integers $h,q$, we say that a graph $G$ has \emph{$(h,q)$-growth} if for any vertex $v \in V(G)$ and any neighbor $u$ of $v$, the number of vertices $w$ satisfying $\dist(w,v)=h$ and $\dist(w,u) \ge h$ is at least $q$. We observe that a graph of minimum degree $\delta$ has $(1,\delta - 1)$-growth. The following theorem, stated in terms of growth, extends Theorem \ref{thmWeight}. The proof of the theorem uses ideas similar to Frankl's method in \cite{FranklGirth}. 

\begin{theorem}
Let $t\geq 1$, $h \ge1$, and $q\ge1$ be integers, and let $G$ be a graph of girth $g \geq 4h(t+1) - 3$ and with $(h,q)$-growth. Then $c(G) > \frac{1}{e\,t}\, q^{t}$. 
\label{thmGrowth}
\end{theorem}

\begin{proof}
If $t = 1$, then a method of Frankl shows that a graph with girth $8h - 3$ and $(h,q)$-growth has cop number greater than $q$ \cite{FranklGirth}. While Frankl only shows that the cop number is greater than $(\delta - 1)^h$, where $\delta$ is the graph's minimum degree, this more general bound in terms of growth follows immediately from his method. Hence, we assume that $t \geq 2$.

We show that under these girth and growth conditions, the robber has a winning strategy, provided that the number of cops is at most $\frac{1}{e\,t} q^{t}$. We observe that our girth condition implies that for any two vertices $u,v$ at distance at most $2h(t+1)-2$, there is a unique geodesic joining $u$ to $v$. 

We begin the game on the robber's move with all cops on a single vertex $y_1 \in V(G)$. We let the robber begin at a vertex $v_1$ adjacent to $y_1$. To show that the cops have no winning strategy on $G$, it suffices to show that the cops cannot win from this position. At any given state $s$ of the game with robber to move, we assume that the robber is at a vertex $v_s$, and that on the previous move, the robber occupied a vertex $y_s$ (except when $s = 1$, in which case $y_1$ is defined separately). At this point, we let the robber select a geodesic path $P_s$ of length $h$ from $v_s$ to a new vertex $v_{s+1}$ that does not pass through $y_s$. The growth condition implies that there are at least $q$ candidates for the vertex $v_{s+1}$. The robber and cops will then take turns moving, making a total of $h$ moves each, and on each move, the robber will move along the path $P_s$ toward $v_{s+1}$. After the robber and cops have each made $h$ moves, we reach the state $s+1$.  We will show that the robber is able to avoid capture during state $s$, and hence forever.

We will assume a stronger condition at the beginning of each state of the game. We will assume that at the beginning of each state $s$, no cop is positioned at a vertex $c$ for which $\dist(c,v_{s}) \leq 2h-2$, except when the geodesic from $c$ to $v_{s}$ passes through $y_s$. The intuition behind this condition is that we will always let our robber move away from $y_s$, so a cop whose geodesic to $v_s$ passes through $y_s$ poses no immediate threat to the robber due to the large girth of $G$. It follows from this assumption that no cop at a vertex $c$ with distance less than $2h-1$ to the robber at the beginning of state $s$ can decrease its distance to the robber while the robber is moving along $P_s$ toward a vertex $v_{s+1}$. All other cops will have a distance of at least $2h - 1$ from $v_s$. This ensures that the robber will always be able to travel along $P_s$ and reach $v_{s+1}$ before being captured. (Here we do not exclude the possibility that the robber is captured when he arrives to $v_{s+1}$.) We have chosen our game's initial configuration to satisfy this condition, and we will show by induction that this condition may be satisfied at the end of each state $s$.

We let $r=(1-\frac{1}{t})q$, and we let $K \leq \frac{1}{e\,t} q^{t}$ denote the number of cops.

Suppose that the game is at state $s\geq 1$. Let $u_1,\dots, u_q$ be distinct vertices at distance exactly $h$ from $v_s$ and at distance at least $h$ from $y_{s}$. (There may be additional such vertices, but we do not need them for our strategy.) For each $i=1,\dots,q$, let ${\mathcal C}_i$ be the set of cops $C$ such that $\dist(C,v_s) \le 2h(t+1)-2$ and $\dist(C,u_i) = \dist(C,v_s) - h$. For such a cop $C\in \bigcup_{i = 1}^q {\mathcal C}_i$, we define its \emph{weight} 
$$w(C) = r^{t+1-\lceil \frac{\rho+2}{2h}\rceil}$$
where $\rho = \dist(C,v_s)$.
Note that each cop in some set ${\mathcal C}_i$ has a unique geodesic path to $v_s$, and that geodesic passes through $u_i$, but not through $y_s$. Thus, ${\mathcal C}_1,\dots,{\mathcal C}_q$ are pairwise disjoint. 
We let $k$ be the number of cops that are in none of the sets ${\mathcal C}_1,\dots,{\mathcal C}_q$, and we define the weight of each such cop to be equal to $1$. Let $W_i = \frac{k}{q} + \sum_{C\in {\mathcal C}_i} w(C)$, and let $W$ be the total weight of all cops. Then $W = \sum_{i=1}^q W_i$. 

The robber selects the vertex $u_j$ for which $W_j$ is minimum and moves to $u_j$ in $h$ moves. As each cop $C$ whose geodesic to $v_s$ does not pass through $y_s$ is at a distance of at least $2h-1$ from $v_s$, the robber will not be captured before reaching $u_j$. On the other hand, for any cop $C'$ whose geodesic to $v_s$ passes through $y_s$, $C'$ will not be able to capture the robber within $h$ moves due to the girth condition of $G$. Furthermore, if $W_j < r^{t-1}$ when the robber is at $v_s$, then no cop $C \in \mathcal{C}_j$ is within distance $4h - 2$ of $v_s$. Therefore, the robber may safely reach $v_{s+1} = u_j$, and after reaching $v_{s+1}$, no cop $C$ will satisfy $\dist(C,v_{s+1}) \leq 2h - 2$, except when the geodesic from $C$ to $v_{s+1}$ passes through $y_{s+1}$ (the vertex adjacent to $v_{s+1}$ on $P_s$). Hence, in order to prove that no cop $C$ ever comes within distance $2h - 2$ of a vertex $v_s$, except when the geodesic from $C$ to $v_s$ passes through $y_s$, it suffices to prove that $W < qr^{t-1}$, since in that case $W_j < r^{t-1}$ for some $1 \leq j \leq q$.

Initially, ${\mathcal C}_1,\dots,{\mathcal C}_q$ are all empty (since all cops are at $y_1$ and the girth of $G$ is greater than $2h(t+1) + h - 1$), and hence all cops have weight $1$. Therefore, as in the method of (\ref{eq:1}),  $W < qr^{t-1}$.

The rest of the proof is by induction. Let us assume that at the current state $s$ we have $W < qr^{t-1}$, and let $W'$ denote the total weight after the robber has moved to $v_{s+1}=u_j$ and state $s$ has ended. For every cop $C \in {\mathcal C}_j$, the robber moved closer to $C$ exactly $h$ times, and $C$ may have moved closer to the robber $h$ times; thus the distance $\rho$ between $C$ and the robber may have decreased by at most $2h$. Consequently, the new weight $w'(C)$ is at most $rw(C)$. For each cop $C \not \in \bigcup_{i = 1}^q \mathcal{C}_i$, the shortest path from the new position of $C$ to $v_{s+1}$, not passing through $y_{s+1}$, is of length at least $2ht - 1$. Hence, the new weight of $C$ is still $w'(C) = 1$. Let us finally consider a cop $C\in {\mathcal C}_i$, where $i\ne j$. By the girth condition of $G$, at the beginning of state $s +1$, a shortest  path from $C$ to $v_{s+1}$ not passing through $y_{s+1}$ is of length $\rho \geq 4h(t+1)-3 - (2h(t+1)-2+ 2h)=2ht-1$, since during state $s$ the cop $C$ could move at most $h$ steps towards $v_{s+1}$. Therefore  the cop $C$ has a new weight $w'(C) = 1$. 

Therefore, we obtain that $W'\leq  r\cdot W_j + K < qr^{t-1}$ by using the method of (\ref{eq:W'}). This completes the proof.
\end{proof}

Theorem \ref{thmGrowth} may be applied in certain situations in which Theorem \ref{thmWeight} is not useful, namely when graphs of high girth have fast growth but low minimum degree. An example of such a graph may be obtained from a graph of high girth and high minimum degree by subdividing each edge approximately the same number of times.

\section{Lower bounds for digraphs}
\label{sectionDirected}

The aim of this section is to show that the techniques of the previous section for undirected graphs may also be applied to digraphs. In order to do so, we will define the \textit{dispersion} of a digraph, which is, in some sense, a directed counterpart of the girth of an undirected graph. Roughly speaking, we do not want short cycles in the underlying undirected graph that are composed of at most four directed geodesic paths, each of which is short. The precise definition of dispersion requires some preparation.

Let $G$ be a digraph. By $vu\in E(G)$ we denote the directed edge from $v$ to $u$. A subgraph of $G$ consisting of two oppositely directed edges $uv,vu\in E(G)$ is called a \emph{digon}. For vertices $v,u\in V(G)$, we define the \emph{distance from $v$ to $u$}, denoted by $\dist(v,u)$, as the length of a shortest directed path from $v$ to $u$. 
A directed path from $v$ to $u$ is \emph{geodesic} if its length is equal to $\dist(v,u)$. A subdigraph of $G$ consisting of two internally disjoint geodesic paths $P$ and $Q$ is a \emph{$(v,u)$-trap} if $P$ is a directed $(v,x)$-path with at least one edge, $Q$ is a directed $(u,x)$-path and $\Vert Q\Vert\le \Vert P\Vert$ (by $\Vert\cdot\Vert$ we denote the length of the path). The intercept $x$ is called the \emph{tip} of the trap. The length of $P$ is called the \emph{length} of the trap. Note that $\Vert P\Vert = \dist(v,x)$ and $\Vert Q\Vert = \dist(u,x)$. We allow for $u=v$, but in this case we require that $P$ and $Q$ each have at least two edges. Note that in this case, $P$ and $Q$ have the same length. We also allow that $u=x$. 

Let $t\ge 1$ be an integer. We say that a digraph $G$ is \emph{$t$-dispersed} if the following conditions hold:
\begin{enumerate}
    \item For every $v,u\in V(G)$ ($v\ne u$), $G$ has no two internally disjoint $(v,u)$-traps, each of length at most $t$.
    \item If $uv\in E(G)$, then $G$ has no $(v,u)$-traps of length at most $t$.
\end{enumerate}

Forbidden subdigraphs used in the definition of $t$-dispersed property are illustrated in Figure \ref{fig:dispersed}. Note that this property forbids certain cycles in $G$ that are composed of 2, 3, or 4 geodesic directed paths of restricted length.

\begin{figure}
\centering
\includegraphics[width=0.85\textwidth]{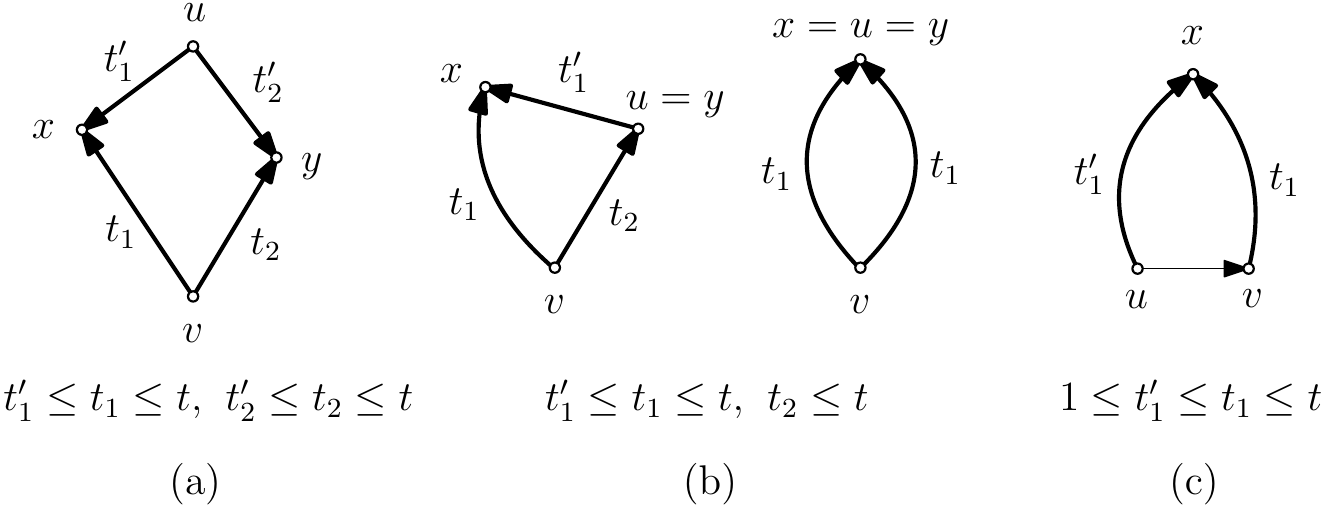}
\caption{Excluded traps in $t$-dispersed graphs. All thick paths shown are geodesics.  (a) Internally disjoint $(v,u)$-traps of lengths $t_1,t_2\le t$. (b) The degenerate versions when $u$ coincides with one or both tips of the trap. (c) $(v,u)$-trap of length $t_1\le t$ when $uv\in E(G)$.}
  \label{fig:dispersed}
\end{figure}

\begin{lemma}\label{lem:trap1}
Let $G$ be a digraph that is $t$-dispersed, and let $v,x\in V(G)$ be vertices with $\dist(v,x)\le t$. Then there is a unique $(v,x)$-geodesic, and if $uv\in E(G)$, then $\dist(u,x)>\dist(v,x)$.
\end{lemma}

\begin{proof}
Suppose that there are two such geodesics $P$ and $Q$. Then it is easy to see that $P\cap Q$ contains two vertices $a,b$ such that the segments on the paths from $a$ to $b$ are internally disjoint (and of length at least 2 since we do not have double edges). These segments would form two internally disjoint $(a,b)$-traps of length at most $t$, contradicting property (1) from the definition of $t$-dispersed. Note that in the two traps, the vertex $b$ is the tip of both traps.

Suppose now that $uv\in E(G)$ and that $\dist(u,x)\le\dist(v,x)$. Consider a $(u,x)$-geodesic $Q$ and let $y$ be the first vertex on it that intersects the $(v,x)$-geodesic $P$. Note that, since we have geodesics, $Q$ does not intersect $P$ at any other vertex between $v$ and $y$. Thus, the union of both geodesics from $v$ to $y$ and from $u$ to $y$ would be a trap of length at most $t$, contrary to property (2) from the definition of $t$-dispersed.
\end{proof}

\begin{lemma}\label{lem:trap2}
Let $G$ be a digraph that is $t$-dispersed, and let $u,v$ be distinct vertices in $G$. Then all $(v,u)$-traps of length at most $t$ contain the same outneighbor of $v$. 
\end{lemma}

\begin{proof}
Suppose that there are two $(v,u)$-traps with geodesics $P\cup Q$ and $P'\cup Q'$, respectively, where the first edge of $P$ and the first edge on $P'$ are different. We may assume that these traps are selected to be of minimum possible lengths and that $u$ is selected so that the lengths of $Q$ and $Q'$ are also minimum. Let $x$ and $x'$ be the tips of the two traps, respectively. Since every subpath of $P$ and $P'$ is a geodesic, Lemma \ref{lem:trap1} implies that $P\cap P' = \{v\}$. Similarly, we see that $Q\cap Q' = \{u\}$ by using the fact that $u$ is selected so that the lengths of $Q$ and $Q'$ are smallest possible. 

Suppose that $Q$ intersects $P'$ in a vertex $u'$. Since $P\cup Q$ is a trap, we have $\Vert P\Vert \ge \Vert Q\Vert = \dist(u,u') +\dist(u',x)$. On the other hand, $\dist(v,u')+\dist(u',x) > \Vert P\Vert$ by Lemma \ref{lem:trap1}. This implies that $\dist(u,u') < \dist(v,u')$. Our minimality choice of traps now implies that $u'=x'$ and also that $u=x'$. By repeating this argument, we end up with the conclusion that $Q$ is internally disjoint from $P'$. Similarly we see that $P$ and $Q'$ are disjoint. Then the two traps are internally disjoint, and their existence contradicts the assumption that $G$ is $t$-dispersed.
\end{proof}

We also define the \emph{trap distance} from $v$ to $u$ as 
$$
    \rho^*(v,u) = \min \{\ell \mid \exists \hbox{ a } (v,u)\hbox{-trap\ of\ length\ } \ell \}.
$$
Observe that $\rho^*(v,u) \le \dist(v,u)$ since any geodesic from $v$ to $u$ is a trap (with the second path of length 0). There is also a way to define the trap distance using distances. Define $S_{\rho}(v)$ (the \emph{sphere of radius $\rho$} around $v$) to be the set of vertices at distance exactly $\rho$ from $v$. Similarly, we define $B_{\rho}(v)$ (the \emph{ball of radius $\rho$} around $v$) to be the set of vertices at distance at most $\rho$ from $v$.

\begin{lemma}\label{lem:sphere intersects ball}
The trap distance $\rho^*(v,u)$ is equal to the minimum value $\rho$ for which $S_\rho(v) \cap B_\rho(u) \neq \emptyset$.
\end{lemma}

\begin{proof}
If there is a $(v,u)$-trap of length $\rho$, then $S_\rho(v) \cap B_\rho(u) \neq \emptyset$. Thus, it suffices to prove that whenever $S_\rho(v) \cap B_\rho(u) \neq \emptyset$, there is a trap of length at most $\rho$. Take a vertex $x$ in the intersection and let $P$ and $Q$ be geodesic paths from $v$ and from $u$ to $x$, respectively. If these paths are disjoint, then we have a trap as claimed. If not, then we replace $x$ by the vertex $y$ in $P\cap Q$ that is as close as possible to $v$. Then we also replace the paths $P$ and $Q$ by the subpaths ending at $y$. Since both paths are geodesic, the number of edges that were removed on $P$ and $Q$ is the same. Thus, the new paths (which are now internally disjoint) form a $(v,u)$-trap of length at most $\rho$.
\end{proof}

\begin{lemma}\label{lem:rho*decreases}
Let $G$ be a digraph that is $t$-dispersed, let $u,v$ be distinct vertices in $G$ and let $u'$ and $v'$ be their outneighbors, respectively. If $\rho^*(v',u')\le t$, then $\rho^*(v,u) \le \rho^*(v',u') + 1$. 
\end{lemma}

\begin{proof}
Let $t_1=\rho^*(v,u)$ and $t_1'=\rho^*(v',u')\le t$.
Suppose, for a contradiction, that $t_1 \ge t_1'+2$. Let $P'\cup Q'$ be a trap certifying that $t_1'=\rho^*(v',u')$. Let $P=vv'P'$ and $Q=uu'Q'$. If both of these paths are geodesic paths, then by Lemma \ref{lem:sphere intersects ball}, $\rho^*(v,u)\le t_1'+1$, which contradicts the assumption that $t_1 \ge t_1'+2$. Consequently, one of the paths is not geodesic. Suppose this is $P$. In that case, $\dist(v',x) = t_1' \le t$, which is not possible by second part of Lemma \ref{lem:trap1}. This contradiction completes the proof.
If $Q$ is not geodesic, the proof is the same.
\end{proof}

If the girth of the underlying undirected graph of a digraph $G$ is at least $4t+1$, then $G$ is $t$-dispersed. Therefore, the following theorems, which give lower bounds on the cop number of a digraph in terms of its dispersion, may be seen as generalizations of the bounds from the previous section.
 
\begin{theorem}
Let $t \geq 1$ be an integer, and let $G$ be a digraph that is $t$-dispersed. For each vertex $v$, let $q_v$ be equal to the outdegree $d^+(v)$ of $v$ if $v$ is not contained in any digons, and be equal to $d^+(v)-1$ otherwise. If $q=\min\{q_v\mid v\in V(G)\}$, then $c(G) > \frac{1}{e\,t} q^{t}$.
\label{thmOutDeg}
\end{theorem}

\begin{proof}
We first consider the case that $t = 1$. Suppose the robber occupies a vertex $v$. As $G$ is $1$-dispersed, a single cop cannot guard more than one out-neighbor of $v$. Therefore, the cop number of $G$ must be at least the minimum out-degree of $G$; otherwise, the robber will always have a safe out-neighbor to visit. The minimum out-degree of $G$ is greater than $\frac{1}{e}q$, so the theorem holds for $t = 1$.

Suppose that $t \geq 2$. We show that under the stated dispersion and degree conditions, the robber has a winning strategy, provided that the number of cops is at most $\frac{1}{e\,t} q^{t}$. 

We begin the game on the robber's move with all cops on a single vertex $y_1 \in V(G)$. We let the robber begin at a vertex $v_1$ which is an out-neighbor of $y_1$. To show that the cops have no winning strategy on $G$, it suffices to show that the cops cannot win from this position. At any given state $s$ of the game with robber to move, we assume that the robber is at a vertex $v_s$ and that he will move to a neighboring vertex $v_{s+1}$ different from $v_{s-1}$ and different from $v_s$. After the cops move, we will reach the next state $s+1$.  We will show that the robber can move to a vertex $v_{s+1}$ for which no cop is positioned at a vertex of the closed in-neighborhood $N^-[v_{s+1}]$. Thus the robber will be able to avoid capture during state $s$, and hence forever. 

As in the proof of Theorem \ref{thmGrowth}, we set $r=(1-\frac{1}{t})q$ and let $K \leq \frac{1}{e\,t} q^{t}$ denote the number of cops. 

Suppose that the game is at a state $s\geq 1$. Let $u_1,\dots, u_q$ be distinct out-neighbors of $v_s$ that are different from $v_{s-1}$. (There may be additional out-neighbors of $v_s$, but we do not need these vertices for our strategy.) For each $i=1,\dots,q$, let ${\mathcal C}_i$ be the set of cops $C$ such that 
\begin{equation}
\label{bobo}
 S_{\rho}(v_s) \cap S_{\rho-1}(u_i) \cap B_{\rho}(C) \neq \emptyset \textrm{ for some } \rho \leq t. 
\end{equation}
Note that each cop in $\mathcal{C}_i$ has the trap distance from $v_s$ bounded by $t$, $\rho^*(v_s,C)\le t$, by Lemma \ref{lem:sphere intersects ball}. By Lemma \ref{lem:trap2}, each cop belongs to at most one family $\mathcal{C}_i$, i.e. $\mathcal{C}_i\cap \mathcal{C}_j = \emptyset$ if $i\ne j$. 

For a cop $C\in \bigcup_{i = 1}^q{\mathcal C}_i$, we let $\rho(C)=\rho^*(v_s,C)$, and we define the \textit{weight} $w(C) = r^{t - \rho(C)}$. We let $k$ be the number of cops that are in none of the sets ${\mathcal C}_1,\dots,{\mathcal C}_q$, and we define the weight of each such cop to be equal to $1$. Let $W_i = \frac{k}{q} + \sum_{C\in {\mathcal C}_i} w(C)$, and let $W$ be the total weight of all cops. Then $W = \sum_{i=1}^q W_i$. Now, the robber selects the vertex $u_j$ for which $W_j$ is minimum and moves to $u_j$. If $W_j < r^{t-1}$, then neither $u_j$ nor any of its in-neighbors contains a cop, since such a cop $C$ would be in $\mathcal{C}_j$ with $\rho(C)=\rho^*(v_s,C)=1$ and would have weight $r^{t-1}$.
Hence, $u_j$ is a safe vertex for the robber. In order to prove that the robber is never captured, it suffices to prove that $W < qr^{t-1}$, since in that case $W_j < r^{t-1}$ for some $1 \leq j \leq q$.

Initially, ${\mathcal C}_1,\dots,{\mathcal C}_q$ are all empty. To see this, note that, initially, all cops are at $y_1$, which is an in-neighbor of $v_1$. 
If a cop $C$ at $y_1$ belongs to some $\mathcal{C}_i$ ($1\le i\le q$), then there is a $(v_1,y_1)$-trap $P\cup Q$ of length at most $t$, where $P$ uses the edge $v_1u_i$. Such a trap would contradict property (2) from the definition of $t$-dispersion. Thus, we have $W < qr^{t-1}$ at the start of the game, as long as we have fewer than $qr^{t-1}$ cops, which we see from the method of (\ref{eq:1}).

The rest of the proof is by induction. Let us assume that at the current state $s$, we have $W < qr^{t-1}$, and let $W'$ denote the total weight after the robber has moved to $v_{s+1}=u_j$ and state $s$ has ended. For every cop $C \in {\mathcal C}_j$, the robber moved closer to $C$, and $C$ may have moved closer to the robber. 
Lemma \ref{lem:rho*decreases} shows that the trap distance between the robber and $C$ can decrease by at most 1, and thus, the new weight $w'(C)$ is at most $rw(C)$. 

Consider now a cop $C'\notin \mathcal{C}_j$. At state $s+1$, the cops are again partitioned into the sets $\mathcal{C}_1', \dots, \mathcal{C}_q'$, defined with respect to the outgoing edges $v_{s+1}u_1', \dots, v_{s+1}u_q'$, where the out-neighbors $u_1',\dots,u_q'$ are distinct from $v_s$. If a cop $C'\notin \mathcal{C}_j$ is in some $\mathcal{C}_i'$, and its weight is $w'(C') > 1$, then $\rho^*(v_{s+1},C') \le t-1$. Lemma \ref{lem:rho*decreases} shows that the trap distance between the robber and $C'$ could have decreased by at most 1 during the previous move. Thus, $\rho^*(v_s,C')$ was at most $t$ at state $s$ and the trap certifying this included the edge $v_sv_{s+1}=v_s u_j$. This shows that $C'$ was in $\mathcal{C}_j$ at state $s$, a contradiction. We conclude that any cop $C$ that was not in $\mathcal{C}_j$ at state $s$ has a new weight $w'(C)=1$ at the beginning of state $s+1$.

Now, we see that $W' \leq r \cdot W_j + K < qr^{t-1}$ in the same way as in (\ref{eq:W'}). This completes the proof.
\end{proof}

Similarly to Theorem \ref{thmGrowth}, we may define a notion of growth for digraphs to generalize Theorem \ref{thmOutDeg} to digraphs that grow quickly but do not necessarily have high minimum out-degree. Given positive integers $h,q$, we say that a digraph $G$ has \emph{$(h,q)$-growth} if for any vertex $v \in V(G)$ and any in-neighbor $y$ of $v$, the number of vertices $w$ satisfying $\dist(v,w)=h$ and $\dist(y,w) \ge h$ is at least $q$. The last condition just says that the geodesic from $v$ to $w$ does not pass through $y$. We observe that a digraph $G$ has $(1,q)$-growth if $q$ is defined in the same way as in Theorem \ref{thmOutDeg}.

\begin{theorem}
Let $h \ge1$ and $t\geq 1$ be integers, and let $G$ be a digraph that is $(h(t+1)-1)$-dispersed and has $(h,q)$-growth. 
Then $c(G) > \frac{1}{e\,t} \, q^{t}$.
\label{thmFrankl}
\end{theorem}

\begin{proof}
First, we consider the case that $t \geq 2$. The theorem is trivial for $q = 1$; therefore we assume that $q \geq 2$. We will show that under the stated dispersion and growth conditions, the robber has a winning strategy, provided that the number of cops is at most $\frac{1}{e\,t} q^{t}$. 

We begin the game on the robber's move with all cops at a single vertex $y_1 \in V(G)$. We choose an out-neighbor $v_1 \in N^+(y_1)$ and let the robber begin the game on $v_1$. To show that the cops have no winning strategy on $G$, it suffices to show that the cops cannot win from this position.

At the beginning of a given state $s \geq 1$ of the game with robber to move, we assume that the robber is at position $v_s$, and that on the previous move, the robber occupied a vertex $y_s$ (except when $s=1$, in which case we have defined $y_1$ separately). During state $s$, the robber will move along some geodesic path $P_s$ of length $h$ which does not pass through $y_s$, from $v_s$ to a new vertex $v_{s+1}$. The robber and cops will take turns moving, making a total of $h$ moves each. One essential difference from the previous proofs is that the path $P_s$ is not selected in advance, but is constructed one vertex at a time, depending on the moves of the cops. After the robber and cops have each made $h$ moves, we reach the state $s+1$.  We will show that the robber is able to avoid capture during state $s$, and hence forever. 

We will use a stronger condition at the beginning of each game state. We will assume that at the beginning of each state $s$, for each vertex $c$ with a cop positioned at it, every $(v_s,c)$-trap of length at most $h-1$ passes through $y_s$. Again, the intuition behind this condition is that we will always let the robber move away from $y_s$, so any traps containing $y_s$ are useless for the cops. Furthermore, we would like for the robber
to be able to move along the geodesic path $P_s$ and reach $v_{s+1}$ before being captured, and this restriction on traps ensures that this will be possible. Our game's initial configuration satisfies this condition, and we will show by induction that this condition is satisfied for the entire duration of each state $s$ (and hence at the beginning of state $s+1$). 

We let $r=(1-\frac{1}{t})q$, and let $K \leq \frac{1}{e\,t} q^{t}$ denote the number of cops. 
We will proceed in a similar way as in previous proofs, but there are some subtle (yet essential) differences, and thus we include the details. Throughout this proof we use the trap distance $\rho^*(v,C)$ from a vertex $v$ and a cop $C$, which refers to the current position of $C$. Since this position is changing over time, we have to specify at which time this trap distance is measured. We always assume that the trap distance is measured at the time when the robber is at the vertex $v$ and it is the robber's move.

Suppose that the game is at the beginning of state $s\geq 1$, and it is the robber's move. Recall that the robber occupies the vertex $v_s$. Let $u_1,\dots, u_q$ be distinct vertices for which $\dist(v_s,u_i) = h$ and $\dist(y_s,u_i) \ge h$ ($1 \leq i\leq q$). These vertices exist by the growth condition of $G$. Consider the out-neighbors $w_1,\dots,w_p$ of $v_s$ that are on some $(v_s,u_i)$-geodesic. For $l=1,\dots,p$, let $d_l\ge1$ be the number of vertices $u_i$ for which the $(v_s, u_i)$-geodesic includes the edge $v_sw_l$. For each $l=1,\dots,p$, let ${\mathcal C}_l$ be the set of cops $C$ for which there is a $(v_s,C)$-trap of length at most $h(t+1)-1$ that uses the edge $v_sw_l$. Since $G$ is $(h(t+1)-1)$-dispersed, Lemma \ref{lem:trap2} implies that each cop belongs to at most one set $\mathcal{C}_l$.

For a cop $C\in \bigcup_{l = 1}^p {\mathcal C}_l$, we let $\rho(C)=\rho^*(v_s,C)$, and we define the \textit{weight} 
\begin{equation}
   w(C) = r^{t+1 - \lceil \frac{\rho(C)+1}{h} \rceil}. \label{eq:weight1}
\end{equation}
We let $k$ be the number of cops that are in none of the sets ${\mathcal C}_1,\dots,{\mathcal C}_p$, and we define the weight of each such cop to be equal to $1$. Let $W_l = \sum_{C\in {\mathcal C}_l} w(C)$, and let $W$ be the total weight of all cops. Then $W = k + \sum_{l=1}^p W_l$. State $s$ will consist of $h$ steps $i$, $i = 1,\dots, h$. If, at the beginning of a step $i$, the robber occupies a vertex $v$ and there exists a $(v,C)$-trap of length at most $h(t+1) - i$ for some cop $C$, then we say that $C$ is \emph{threatening}. Note that at the beginning of step $1$ of the current state $s$, each cop in $\bigcup_{l = 1}^p {\mathcal C}_l$ is threatening. We define the ``average weight" of the threatening cops at step~1:
$$
   Z := \tfrac{1}{q} \sum_{l=1}^p W_l.
$$
Now, at the beginning of step $1$, the robber selects the vertex $w_j$ for which $W_j/d_j$ is minimum, and the robber moves to $w_j$. Let us note that $W_j/d_j \le Z$. If not, then $W_l>Zd_l$ for every $l=1,\dots,p$, and hence $W-k = \sum_{l=1}^p W_l > Z\sum_{l=1}^p d_l = Zq = W-k$, a contradiction. Now we let the cops move. After the cops move, only the cops in $\mathcal{C}_j$ remain threatening. Indeed, if a cop $C$ is threatening at the beginning of step 2, then there exists a $(w_j,C)$-trap of length at most $h(t+1) - 2$. Then, Lemma \ref{lem:rho*decreases} tells us that at the beginning of step 1, there existed a $(v_s, C)$-trap of length at most $h(t+1)-1$, and by Lemma \ref{lem:trap2}, this trap must have included the edge $v_s w_j$. This tells us that $C \in \mathcal{C}_j$.

Now, during steps $2$ to $h$, we repeat this process without changing $u_1,\dots,u_q$ or the weights $w(C)$, but updating $p$ and the robber's out-neighbors $w_1,\dots,w_p$, the corresponding sets $\mathcal{C}_l$, and the numbers $d_l$. Again, the robber moves to the out-neighbor $w_j$ for which $\sum_{C\in \mathcal{C}_j}w(C)/d_j$ is minimum. By the same arguments that we have made previously, only the cops in $\mathcal{C}_j$ remain threatening, and $\sum_{C\in \mathcal{C}_j}w(C)/d_j \leq Z$ still holds. In this way, the robber moves away from $v_s$ and reaches one of the vertices $u_i$ after making $h$ steps. As $\sum_{C\in \mathcal{C}_j}w(C)/d_j \leq Z$ holds at each step, and as $d_j = 1$ at beginning of step $h$, when the robber moves to a vertex $u_i$, the total weight of the threatening cops after step $h$ is at most $Z$. Then, we set $v_{s+1}=u_i$ and begin the new state $s+1$. 

We observe that for a cop $C$, if $\rho^*(v_s,C)$ was at most $h-1$ at the beginning of state $s$, then $C$ will not be able to capture the robber within the next $h$ moves, since the only short traps pass through $y_s$, and the robber moves away from $y_s$ at each step of state $s$. Each cop that is still threatening when the robber arrives at $v_{s+1}$ had weight at most $Z$ when the robber was at $v_s$. If $Z < r^{t-1}$ when the robber is at $v_s$, then each such cop $C$ had $\rho^*(v_s,C) \geq  2h$. Hence, at the end of state $s$, the trap distance from the robber to any such cop will be at least $h$ by Lemma \ref{lem:rho*decreases}. It therefore suffices to prove that $W < qr^{t-1}$, since in that case $Z < r^{t-1}$.

Initially, ${\mathcal C}_1,\dots,{\mathcal C}_p$ are all empty by condition (2) of our definition of dispersion. Thus $W < qr^{t-1}$ as long as we have fewer than $qr^{t-1}$ cops, which we prove as in (\ref{eq:1}).

The rest of the proof is by induction. Let us assume that at the current state $s$ we have $W < qr^{t-1}$. Let $\mathcal{D}$ be the set of threatening cops that have remained threatening throughout all $h$ steps of state $s$. We consider the new weight $w'(C)$ of each cop $C$ at the start of state $s+1$.
We claim that at the beginning of state $s+1$, each cop $C$ with weight greater than $1$ belongs to $\mathcal{D}$. Indeed, suppose that a cop $C$ has weight greater than $1$ at the beginning of state $s+1$. By our definition (\ref{eq:weight1}) of weight, this implies that $\rho^*(v_{s+1}, C) \leq ht - 1$. Then, by Lemma \ref{lem:rho*decreases}, we see that $\rho^*(v_s,C) \leq h(t+1) - 1$, and hence $C$ belongs to $\mathcal{D}$.

Moreover, for each cop $C \in \mathcal{D}$, the trap distance between the robber and $C$ has decreased by at most $h$. Therefore, for a cop $C \in \mathcal{D}$, at the beginning of state $s+1$, $C$ has a new weight $w'(C)$, which is at most $rw(C)$. Note that we have shown that the total weight from the beginning of state $s$ of all cops in $\mathcal{D}$ is at most $Z$, and hence the new total weight of the cops in $\mathcal{D}$ at the beginning of state $s+1$ is at most $rZ$. Furthermore, as shown above, any cop that does not belong to $\mathcal D$ has weight $1$ at the beginning of state $s+1$. Thus, at the beginning of state $s+1$, the new total weight $W'$ satisfies
$$W' \le rZ + K < r \cdot \frac{W}{q} + \frac{qr^{t-1}}{t} < qr^{t-1},$$
with the last inequality proved in the same way as $(\ref{eq:W'})$. This completes the proof for the case that $t \geq 2$.

In the case that $t = 1$, a very similar argument may be used. This time, however, we let each cop have a weight $1$ throughout the entire game. We again let the robber begin at a vertex $v_1$ with an in-neighbor $y_1$ at which all cops begin the game. At the beginning of a given state $s \geq 1$, we again suppose that the robber occupies a vertex $v_s$ and that the robber occupied a vertex $y_s$ on the previous move, except when $s = 1$, in which case $y_1$ is already defined. We again aim to show that for any vertex $c$ with a cop positioned at it, every $(v_s, c)$-trap of length $h-1$ passes through $y_s$. By using the same ``average weight" argument, we may show that the average weight $Z$ of the threatening cops may be kept below $1$ and hence that the robber will be able to evade all threatening cops at each stage. As all other parts of the strategy are the same as the $t \geq 2$ case, we do not include the details.
\end{proof}

\section{Bounds for expanders and an upper bound in terms of girth}

In Section \ref{sectionUndirected}, we showed that asymptotically, a graph of girth $g$ and minimum degree $d$ has cop number $\Omega(g^{-1}(d - 1)^{\frac{g}{4}})$. As discussed at the end of the introduction, we believe that the factor $\frac{1}{4}$ in the exponent is best possible. We are not able to prove this, but in this section we prove that this constant factor cannot be made much larger.

We will show that if $\Omega((\delta-1)^{c g})$ is a lower bound for the cop number of all graphs of minimum degree $\delta$ and girth $g$, then $c \leq \frac{3}{8}$. In order to do this, we will present families of $d$-regular graphs of fixed degree $d$ and arbitrarily large girth $g$, whose cop number is at most $(d-1)^{c g}$, where $c$ can be made arbitrarily close to $\frac{3}{8}$. One such family are the bipartite Ramanujan graphs introduced by Lubotsky, Phillips, and Sarnak in \cite{LPS}; these graphs are often called LPS Ramanujan graphs. In order to show that these graphs have cop number as small as claimed, we will present some general bounds for expander graphs and then apply these bounds to the family of LPS Ramanujan graphs. 

In proving our upper bound, we will make extensive use of the following graph parameter.

\begin{definition}
Let $G$ be a graph. For a real number $0 < \gamma < 1$, we define the parameter
$$h_{\gamma}(G) = \min_{0 < |S| \leq n^{1-\gamma}} \frac{|\partial S|}{|S|},$$
where $\partial S$ is the set of vertices in $V(G) \setminus S$ that have a neighbor in $S$.
\label{defIso}
\end{definition}

For a graph $G$, a nonempty vertex set $S \subseteq V(G)$, and a nonnegative integer $r$, we use $B_r(S)$ to denote the set of vertices in $V(G)$ whose distance from $S$ is at most $r$.

\begin{lemma}
Let $G$ be a graph of order $n$, $0 < \gamma < 1$, and let $h_{\gamma}(G) \geq \varepsilon > 0$. Let $r \ge 0$ be an integer. Then for any nonempty vertex set $S \subseteq V(G)$,  
$$|B_r(S)| \geq \min\{n^{1 - \gamma}, |S|(1+\varepsilon)^r\}.$$
\label{lemmaR}
\end{lemma}

\begin{proof}
The result follows easily by induction on $r$.
\end{proof}

The following theorem gives the main tool that we will use to establish our upper bound on $c$. The cop strategy used in this theorem is probabilistic and borrows key ideas from the strategies from \cite{Lu}, \cite{Sudakov}, and \cite{SeyyedThesis}.

\begin{theorem}
Let $\Delta \geq 3$, $\varepsilon$ ($0 < \varepsilon \leq \Delta - 2$), and $\gamma$ ($0 < \gamma \leq \frac{1}{2} (1- \log_{\Delta - 1}(1+\varepsilon))$) be fixed constants. Then there is an integer $n_0$ such that for every graph $G$ of maximum degree at most $\Delta$, with $h_{\gamma}(G) \ge \varepsilon$ and of order $n\ge n_0$, 
$$c(G) \leq n^{1-\frac{1}{2}\log_{\Delta - 1}(1+\varepsilon) + o(1)},$$ 
where the asymptotics are considered with respect to $n$.
\label{thmhgamma}
\end{theorem}

\begin{proof}
We take a small value $\delta > 0$, and define the value 
$\kappa = (\frac{1}{2} - 2 \delta) \log_{\Delta - 1}(1+\varepsilon)$. Note that $\gamma < \frac{1}{2} - \kappa.$  Next, we fix a sufficiently large graph $G$ of maximum degree at most $\Delta$ which satisfies $h_{\gamma}(G) \geq \varepsilon$. We will show that $c(G) \leq n^{1-\kappa+o(1)} = n^{1-\frac{1}{2}\log_{\Delta - 1}(1+\varepsilon) + o(1)}$, with the last equality holding by letting $\delta$ tend to $0$.

We define $r = \lfloor \frac{(\frac{1}{2} - \delta)\log n}{\log (\Delta - 1)} \rfloor$. We will prove later that with this choice of $r$, any vertex $v \in V(G)$ has $|B_r(v)| < \sqrt{n}$. Furthermore, we estimate: 
$$(1+\varepsilon)^r = (\Delta-1)^{r \log_{\Delta-1}(1+\varepsilon)} > n^{(\frac{1}{2} - 2\delta) \log_{\Delta-1}(1+\varepsilon)} = n^\kappa.$$
For the inequality to hold, we have assumed that $n$ is large enough so that $\delta\log_{\Delta-1}n\ge1$.
Lemma \ref{lemmaR} tells us that $|B_r(S)| \geq |S| (1 + \varepsilon)^r$ holds for any set $S \subseteq V(G)$ that is not too large, and therefore $\kappa$ is defined so that 
\begin{equation}
    \label{eq:ball around S large}
    |B_r(S)| \geq |S|n^\kappa 
\end{equation}
for such a set $S$.  

Let $C \subseteq V(G)$ be a random subset of vertices, obtained by taking each vertex of $G$ independently at random with probability $p = n^{-\kappa} \log^3 n$. 
Note that in contrast to previous sections, here the symbol $C$ will denote the set of starting positions of cops, rather than an individual cop.

We place a cop at each vertex of $C$. The size of $C$ is a binomially distributed random variable, and the expected size of $C$ is $n^{1-\kappa} \log^3 n$. We will estimate the size of $C$ with a Chernoff bound (c.f. \cite{MolloyReed}, Chapter 5), which states that for a binomally distributed variable $X$ with mean $\mu$, and for $0 \leq t \leq \mu$, 
\begin{equation}
    \Pr(|X - \mu| > t) < 2\exp(-\frac{t^2}{3 \mu}).
    \label{eq:Chernoff}
\end{equation}

By (\ref{eq:Chernoff}), the probability that $|C| > 2n^{1-\kappa} \log^3 n$ is less than $2 \exp \left ({-\frac{1}{3}n^{1-\kappa}  \log^3 n} \right ) = o(1)$, and hence we see that a.a.s. 
\begin{equation}
    |C| \leq 2n^{1-\kappa} \log^3 n.
    \label{eq:size of C small}
\end{equation}

We let the robber begin the game at a vertex $v \in V(G)$. We will attempt to capture the robber by using the first $r$ moves of the game to place a cop at each vertex of $B_r(v)$. This way, the robber will surely be captured before he has a chance to escape from $B_r(v)$. First, we will estimate the size of $B_r(v)$. As $G$ has maximum degree at most $\Delta$, we see that 
$$
   |B_r(v)| \leq 1 + \Delta \sum_{j = 0}^{r-1} (\Delta - 1)^j = 
   1+\frac{\Delta}{\Delta - 2} ((\Delta - 1)^r - 1) \le
   1 + \frac{\Delta}{\Delta - 2}(n^{\frac{1}{2} - \delta} - 1) < \sqrt{n},
$$
with the last inequality holding for sufficiently large $n$.

We now define an auxiliary bipartite graph $H$ with (disjoint) partite sets $B_r(v)$ and $C$. For vertices $u \in B_r(v)$ and $u' \in C$, we add an edge $uu'$ to $E(H)$ if the distance in $G$ from $u$ to $u'$ is at most $r$. If we can find a matching in $H$ that saturates $B_r(v)$, then the cops from $C$ can fill $B_r(v)$ in the first $r$ moves and capture the robber.

We show that a.a.s., a matching in $H$ saturating $B_r(v)$ exists. Indeed, suppose that no such matching exists. Then by Hall's theorem, there exists a set $S \subseteq B_r(v)$ such that $|B_r(S) \cap C| < |S|$. We denote by $A_S$ the event that $|B_r(S) \cap C| < |S|$, and we calculate $\Pr(A_S)$. The quantity $|B_r(S) \cap C|$ is a binomially distributed random variable with expected value $p|B_r(S)|$. 

By Lemma \ref{lemmaR} and (\ref{eq:ball around S large}), 
$$|B_r(S)| > \min\{|S|(1+\varepsilon)^r, n^{1-\gamma}\} \geq \min \{|S|n^\kappa, n^{1-\gamma}\} = |S|n^\kappa,$$
with the last inequality coming from the facts that $\gamma < \frac{1}{2} - \kappa$, $|S| < \sqrt{n}$, and that $n$ is large enough.
Therefore,  $p|B_r(S)| > |S| \log^3 n$.
By applying (\ref{eq:Chernoff}) with $t = |S|(\log^3 n - 1)$, we can estimate as follows: 
$$\Pr(A_S) < 2 \exp \left ({\frac{-|S|^2(\log^3 n - 1)^2}{3|S| \log^3 n}} \right ) < \exp \left ({-\frac{1}{3}|S| \log^2 n} \right ).$$
Then, by considering the nonempty sets $S \subseteq B_r(v)$ of all sizes $a\le \sqrt{n}$, the probability that no event $A_S$ occurs is at least 
\begin{eqnarray*}
     1 - \sum_{a = 1}^{\sqrt{n}} \binom{\sqrt{n}}{a} e^{-\frac{1}{3}a \log^2 n} &>& 
     1 - \sum_{a = 1}^{\sqrt{n}}  e^{\frac{1}{2}a \log n -\frac{1}{3}a \log^2 n} \\
     &>& 1 - \sum_{a = 1}^{\infty}  e^{-\frac{1}{6}a \log^2 n} \\
     &=& 1 - \frac{e^{-\frac{1}{6} \log^2 n}}{1-e^{-\frac{1}{6} \log^2 n}} = 1 - o(\tfrac{1}{n}).
\end{eqnarray*}
We see that a.a.s.\ no event $A_S$ occurs for the robber's choice of $v$, and as the robber has $n$ choices for his starting vertex, we see furthermore that a.a.s.\ no event $A_S$ occurs regardless of the vertex $v \in V(G)$ at which the robber chooses to begin the game. It follows that a.a.s.\ we have a team of at most $2n^{1-\kappa} \log^3 n$ cops who have a strategy to capture the robber in $r$ moves. Recalling that $\kappa = (\frac{1}{2} - 2 \delta) \log_{\Delta - 1}(1+\varepsilon)$, and letting $\delta$ tend to $0$, we see that there is a winning strategy with $n^{1-\frac{1}{2}\log_{\Delta - 1}(1+\varepsilon)  + o(1)}$ cops.
\end{proof}

The last ingredient we will need is a connection between vertex expansion and the eigenvalues of a regular bipartite graph, which was proven by Tanner (see Theorem 2.1 in \cite{Tanner}).

\begin{lemma}
\label{lem:Tanner}
Let $H$ be a $d$-regular bipartite graph on $n$ vertices with a balanced bipartition $X \cup Y$. Let $B \in {\mathbb R}^{X\times Y}$ be the square matrix whose $(x,y)$-entry ($(x,y)\in X\times Y$) is $1$ if $x$ is adjacent to $y$ and is $0$ otherwise. If the second largest eigenvalue of $BB^T$ is at most $\lambda$, then any set $S \subseteq X$ has at least $f(|S|)$ neighbors in $Y$, where 
$$f(|S|) = \frac{d^2|S|}{\lambda + 2(d^2 - \lambda)|S|/n}.$$
\end{lemma}




\begin{lemma} 
Let $0 < \gamma < 1$ be fixed. Let $G$ be a connected $d$-regular bipartite graph. Suppose that $G$ has eigenvalues $d = \lambda_1(G) \geq \lambda_2(G) \geq \dots \geq  \lambda_n(G) = -d$. Then $h_{\gamma}(G) \geq \Bigl(\frac{d}{\lambda_2(G)}\Bigr)^2 - 1 - o(1)$, where the asymptotics refer to the order of the graph.
\label{lemmaExpand}
\end{lemma}

\begin{proof}
Let $G$ have bipartition $X \cup Y$. We set $\lambda = \lambda_2(G)$ and $\alpha =\lambda^2/d$. If $\alpha = d$, then the lemma is trivial, so we assume that $\alpha < d$. The adjacency matrix of $G$ can be written in the form
\[A(G)=
\left[
\begin{array}{c|c}
0 & B \\
\hline
B^T & 0
\end{array}
\right] ,
\]
where $B$ is defined for $G$ as in Lemma \ref{lem:Tanner}. It follows that
\[A^2(G)=
\left[
\begin{array}{c|c}
BB^T & 0 \\
\hline
0 & B^T B
\end{array}
\right].
\]
Matrices $BB^T$ and $B^TB$ have the same eigenvalues and it is easy to see that they are equal to the squares of the eigenvalues of $G$. In particular, the second largest eigenvalue of each of $BB^T$ and $B^T B$ is equal to $\lambda^2  = \alpha d$. 

Finally, we consider a set $S \subseteq V(G), |S| \leq n^{1 - \gamma}$. By Lemma \ref{lem:Tanner}, $S \cap X$ has at least $ (\frac{d }{\alpha} - o(1))|S \cap X|$ neighbors in $Y$, and by applying Lemma \ref{lem:Tanner} with $X$ and $Y$ exchanged, $S \cap Y$ has at least $ (\frac{d }{\alpha} - o(1))|S \cap Y|$ neighbors in $X$. This implies that 
$$
  |\partial(S \cap X)\setminus S| \ge \Bigl(\frac{d }{\alpha} - o(1)\Bigr)|S \cap X| - |S\cap Y|
$$
and
$$
  |\partial(S \cap Y)\setminus S| \ge \Bigl(\frac{d }{\alpha} - o(1)\Bigr)|S \cap Y| - |S\cap X|.
$$
By summing up these two inequalities, we see that 
$|\partial S| \geq (\frac{d}{\alpha} - 1 -  o(1))|S|$, or equivalently, 

$$\frac{|\partial S|}{|S|} \geq \Bigl(\frac{d}{\lambda_2(G)}\Bigr)^2 - 1 -  o(1).$$
\end{proof}

We are ready to prove our upper bound on the coefficient $c$ of $g$. By letting $p$ tend to infinity in the following theorem, we obtain families of $d$-regular graphs of girth $g$ and cop number at most $(d-1)^{c g}$, with $c$ arbitrarily close to $\frac{3}{8}$.

\begin{theorem}\label{thm:3/8}
Let $p$ be a prime number for which $p \equiv 1 \mod 4$, and let $d=p+1$. There exists an infinite family of $d$-regular graphs $X$ with increasing girth $g$, whose cop number is bounded above by 
$$c(X) \leq p^{ (1 + 2\log_p 4 + o(1))\frac{3}{8}g}.$$
\end{theorem}

\begin{proof}
We use the Ramanujan graphs $X^{p,q}$ discovered by Lubotsky, Phillips, and Sarnak in \cite{LPS}. These graphs are $d$-regular, where $d=p+1$ and $p$ is any prime that $p \equiv 1 \mod 4$. The second parameter $q$ is also a prime that is larger than $\sqrt{p}$, $q \equiv 1 \mod 4$, and such that the Legendre symbol of $q$ and $p$ satisfies $\left( \frac{q}{p} \right) = -1$. By the definition of Ramanujan graphs, their second eigenvalue $\lambda_2$ satisfies $\lambda_2^2 \leq 4(d-1)$.
The graph $X=X^{p,q}$ is a $d$-regular bipartite Cayley graph on $n = q(q^2 - 1)$ vertices, and the girth $g$ of $X^{p,q}$ satisfies $g \geq \frac{4 \log q}{\log p} - 1$ \cite{LPS}.

 
We now compute an upper bound on the cop number $c(X)$ of $X$. By Lemma \ref{lemmaExpand}, $h_{\gamma}(X^{p,q})  \geq \frac{d}{4} - 1 - o(1) > \frac{p+1}{4} - 1 - o(1)$, for any $0 < \gamma < 1$. Then, by applying Theorem \ref{thmhgamma}, 
$$c(X)  \leq n^{1 - \frac{1}{2} \log_p \frac{p+1}{4} + o(1)} \le n^{\frac{1}{2} + \log_p 4 + o(1)}.$$

Finally, we express this bound in terms of the girth and minimum degree of $X$. As $X$ is a $(p+1)$-regular Cayley graph, $\delta(X) - 1 = p$. Furthermore, as $q$ tends to infinity, $g(X) \geq (4+o(1)) \frac{\log q}{\log  p}$. As $n = (1 - o(1))q^3$, it follows that $n \leq p^{(\frac{3}{4} + o(1))g}$, and hence $c(X) \leq p^{ (1 + 2\log_p 4 + o(1))\frac{3}{8}g}$. This completes the proof.
\end{proof}

We may also use Theorem \ref{thmhgamma} to gain upper bounds on cop number in terms of a better known graph parameter and give a sufficient condition for the cop number of a graph on $n$ vertices to be of the form $O(n^{1-\alpha})$ for some constant $\alpha > 0$.
The \emph{isoperimetric number} $h(G)$ of $G$ is defined as follows:
$$h(G) = \min_{0 < |S| \leq n/2} \frac{|\delta S|}{|S|},$$
where $\delta S$ is the set of edges with exactly one endpoint in $S$. (See \cite{Mo-Isoperim} for more details.)

From Theorem \ref{thmhgamma} and Definition \ref{defIso}, we obtain the following corollary.



\begin{corollary}
\label{corWeak}
Let $\mathcal{G}$ be a family of graphs such that for all $G \in \mathcal{G}$, the maximum degree of $G$ is at most $\Delta \geq 3$. Let $0 < \varepsilon \leq \Delta - 2$, and suppose that $h(G) \geq \varepsilon$ for all $G \in \mathcal{G}$ sufficiently large. Then for all $G \in \mathcal{G}$, $c(G) \leq n^{1-\frac{1}{2}\log_{\Delta - 1}(1+\varepsilon/ \Delta)+o(1)}$.
\end{corollary}

Corollary \ref{corWeak} gives an implication about the so-called \emph{Weak Meyniel's Conjecture}. While Meyniel's Conjecture \ref{conjMeyniel} asserts that the cop number of a graph on $n$ vertices is bounded by $O(\sqrt{n})$, its weakened version asserts that there exists a value $\alpha > 0$ such that every graph on $n$ vertices has cop number $O(n^{1-\alpha})$ (see, e.g.~\cite{Baird}). Even this weakened conjecture is still widely open. Corollary \ref{corWeak} implies that the weakened Meyniel's conjecture holds for expanders of bounded degree.

\section{Conclusion}
We have shown that in the optimal lower bound for the cop number of a graph with girth $g$ and minimum degree $\delta$ of the form $\Omega((\delta - 1)^{cg})$, the constant $c$ satisfies $\frac{1}{4} \leq c \leq \frac{3}{8}$. We suspect that $\frac{1}{4}$ is the correct answer, and it would be interesting to reduce or close this gap between the lower and upper bounds of $c$. While Meyniel's conjecture and the folklore conjecture about degree and diameter imply that $\frac{1}{4}$ is indeed correct, it may be possible to show $c = \frac{1}{4}$ without using these conjectures.

\raggedright
\bibliographystyle{abbrv}
\bibliography{bibGirth}

\end{document}